\documentclass[a4paper, 12pt]{article}
\usepackage{mathrsfs}
\usepackage{amsmath}
\usepackage{amsthm}
\usepackage{array}
\makeatletter
\def\th@plain{%
  \upshape 
}
\makeatother

\makeatletter
\renewenvironment{proof}[1][\proofname]{\par
  \pushQED{\qed}%
  \normalfont \topsep6\p@\@plus6\p@\relax
  \trivlist
  \item[\hskip\labelsep
        \bfseries
    #1\@addpunct{.}]\ignorespaces
}{%
  \popQED\endtrivlist\@endpefalse
}
\makeatother

\newtheorem{theorem}{Theorem}
\numberwithin{theorem}{section}
\newtheorem{lemma}{Lemma}

\newtheorem{conjecture}{Conjecture}
\newtheorem*{conjecture*}{Conjecture}

\numberwithin{claim}{theorem}

\theoremstyle{definition}

\newtheorem{remark}{Remark}
\usepackage[inner=3.2cm,outer=3.2cm,top=3.8cm,bottom=3.8cm,marginparsep=1pc,marginparwidth=4.5pc,footskip=3.5pc]{geometry}
\usepackage[T1]{fontenc}
\usepackage[varg]{txfonts}
\usepackage{listings}
\usepackage{graphicx, epsfig, subfigure}
\usepackage{enumerate} 

\usepackage[square, numbers, sort&compress]{natbib} 
\ifx\pdfoutput\undefined
 \usepackage[dvipdfm,%
  pdfstartview=FitH, 
  bookmarks=true,%
  bookmarksnumbered=true, 
  bookmarksopen=true, 
  plainpages=false,%
  pdfpagelabels,%
  colorlinks=true, 
  linkcolor=blue, 
  citecolor=blue,%
  urlcolor=black,
  pdfborder=001]{hyperref}
  \AtBeginDvi{}  
\else
 \usepackage[pdftex,%
  pdfstartview=FitH, 
  bookmarks=true,%
  bookmarksnumbered=true, 
  bookmarksopen=true, 
  plainpages=false,%
  pdfpagelabels,%
  colorlinks=true, 
  linkcolor=blue, 
  citecolor=blue,%
  urlcolor=black,
  pdfborder=001]{hyperref}
\fi

\newcommand{\etal}{et~al.\ }
\newcommand{\ie}{i.e.,\ }

\numberwithin{equation}{subsection}

\def\int(#1){\mathrm{int}(#1)}
\def\ext(#1){\mathrm{ext}(#1)}
\def\Int(#1){\mathrm{Int}(#1)}
\def\Ext(#1){\mathrm{Ext}(#1)}
\newcounter{listitem}
\newenvironment{listitem}%
{\refstepcounter{listitem}\vspace{5pt}\par\noindent
\textbf{(\thelistitem)}
\begin{upshape}\noindent}%
{\end{upshape}\vspace{5pt}\par}

\begin{document}
\title{Every planar graph without adjacent short cycles is $3$-colorable}
\author{%
Tao Wang\thanks{{\tt Corresponding
email: wangtaonk@yahoo.com.cn}}\\[.5em]
\small Institute of Applied Mathematics, \\
\small College of Mathematics and Information Science,\\
\small Henan University, Kaifeng, 475001, P. R. China
}
\date{}
\maketitle

\begin{abstract}%
Two cycles are {\em adjacent} if they have an edge in common. Suppose that $G$ is a planar graph, for any two adjacent cycles $C_{1}$ and $C_{2}$, we have $|C_{1}| + |C_{2}| \geq 11$, in particular, when $|C_{1}| = 5$, $|C_{2}| \geq 7$. We show that the graph $G$ is $3$-colorable.
\end{abstract}

\section{Introduction}
In 1852, Francis Guthrie proposed the Four Color Problem. In 1976,
K. Appel and W. Haken proved the Four Color Theorem:
\begin{theorem}
  Every planar graph is $4$-colorable.
\end{theorem}
In 1976, Garey \etal \cite{Garey1976} proved the problem of deciding
whether a planar graph is $3$-colorable is NP-complete. In 1959,
Gr\"{o}tzsch \cite{Grotzsch1959} showed that every planar graph
without $3$-cycles is $3$-colorable. In 1976, Steinberg conjectured
the following:
\begin{conjecture}[Steinberg's Conjecture]
  Every planar graph without $4$- and $5$-cycles is $3$-colorable.
\end{conjecture}
This conjecture remains open. In 1991, Erd\"{o}s suggested the
following relaxation of Steinberg's Conjecture by asking whether
there exists an integer $k$ such that the absence of cycles of
lengths from $4$ to $k$ in a planar graph guarantees its
$3$-colorability.

Abbott and Zhou \cite{Abbott1991} proved such an integer $k$ exists
and $k \leq 11$. The bound on integer $k$ was later improved to $10$
by Borodin \cite{Borodin1996b}, to $9$ by Borodin
\cite{Borodin1996a} and, independently, by Sanders and Zhao
\cite{Sanders1995}, to $8$ by Salavatipour \cite{Salavatipour2002},
to $7$ by Borodin \etal \cite{Borodin2005}.

Towards Steinberg's Conjecture, one direction is to show that planar
graph without adjacent short cycles is $3$-colorable, for instance,
the following result is such an attempt:
\begin{theorem}\label{Xu_BorodinTheorem}
Every planar graph without $5$- and $7$-cycles and without adjacent
triangles is $3$-colorable.
\end{theorem}

Note that the first attempt to prove this theorem was made by Xu
\cite{Xu2006}, but his proof was not correct. Borodin \etal gave a
new proof of \autoref{Xu_BorodinTheorem}, see \cite{Borodin2009}.

Recent progress are presented the service of theorems below.
\begin{theorem}[Borodin \etal \cite{Borodin2006}]
Every planar graph without triangles adjacent to cycles of length
from $3$ to $9$ is $3$-colorable.
\end{theorem}

\begin{theorem}[Borodin \etal \cite{Borodin2010}]
Every planar graph in which no $i$-cycle is adjacent to a $j$-cycle
whenever $3 \leq i \leq j \leq 7$ is $3$-colorable.
\end{theorem}

\begin{conjecture}[Strong Bordeaux Conjecture \cite{Borodin2003}]
Every planar graph without $5$-cycles and without adjacent triangles
is $3$-colorable.
\end{conjecture}
\begin{conjecture}[Novosibirsk $3$-Color Conjecture, \cite{Borodin2006}]
Every planar graph without $3$-cycles adjacent to $3$-cycles or
$5$-cycles is $3$-colorable.
\end{conjecture}

\section{Preliminaries}

In this paper, the graphs considered may contain multiple edges, but
no loops. The {\em neighborhood} of a vertex $v \in V(G)$, denoted
by $N_{G}(v)$, is the set of all the vertices adjacent to $v$, \ie
$N_{G}(v) = \{u \in V(G) \mid uv \in E(G)\}$. The {\em degree} of a
vertex $v$ in $G$, denoted by $\deg_{G}(v)$, is the number of its
neighbors in $G$, \ie $\deg_{G}(v) = |N_{G}(v)|$. A vertex of degree
$k$ is also referred as a {\em $k$-vertex}. Two cycles are {\em
adjacent} if they have an edge in common.

For a plane graph, the edges and vertices divide the plane into a
number of {\em faces}. The unbounded face is called the {\em outer
face}, and the others are called {\em inner faces}. The boundary of
the outer face of $G$ is called the {\em outer boundary} of $G$ and
denoted by  $C_{0}(G)$. If $C_{0}(G)$ is a cycle, then $C_{0}(G)$ is
called the {\em outer cycle} of $G$. We call a vertex $v$ of $G$ an
{\em outer vertex} of $G$ if $v$ is on $C_{0}(G)$; otherwise $v$ is
an {\em inner vertex} of $G$. Similarly we define an outer edge and
an inner edge of $G$. The {\em degree of a face} $F$ of $G$ is the
number of edges in its boundary, counting those edges twice for
which $F$ lies on both sides. A $k$-face is a face of degree $k$. A
face is said to be {\em incident} with vertices and edges in its
boundary, and two faces are {\em adjacent} if their boundaries have
an edge in common. A vertex is {\em bad} if it is an inner
$3$-vertex and is incident with a triangle. Let $C$ be a cycle of a
plane graph $G$. The cycle $C$ divides the plane into two regions,
the unbounded region is denoted by $\ext(C)$, and the other region
is denoted by $\int(C)$. If both $\int(C)$ and $\ext(C)$ contain at
least one vertex, then we say that the cycle $C$ is a {\em
separating cycle} of $G$. Let $u$ and $v$ be two vertices of a cycle
$C$ in $G$, the segment of $C$ clockwisely from $u$ to $v$ is
denoted by $C[u, v]$, and $C(u, v) = C[u, v] - \{u, v\}$.

A {\em nonadjacency} graph is one whose vertices are labeled by
integers greater than two and each integer appears at most once.
Given a graph $G_{\mathcal{A}}$ of nonadjacency, we say that a graph
$G$ belongs to $G_{\mathcal{A}}$ or $G$ has the nonadjacency
property $\mathcal{A}$ if no two cycles of lengths $i$ and $j$ are
adjacent in $G$ when the vertices labeled with $i$ and $j$ are
adjacent in $G_{\mathcal{A}}$.

\begin{figure}%
\centering
\includegraphics{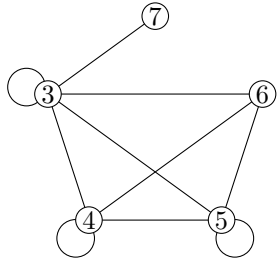}
\caption{The nonadjacency graph $G_{\mathcal{A}}$.}
\label{GA-1}
\end{figure}

Let $\mathcal{G}_{(A)}$ be the class of graphs belongs to the
nonadjacency graph depicted in \autoref{GA-1}. In this paper, we
prove the following.
\begin{theorem}\label{MainResult}
  Every planar graph in $\mathcal{G}_{(A)}$ is $3$-colorable.
\end{theorem}
\section{Proof of the main result}
In attempt to prove \autoref{MainResult}, we prove a strong color
extension lemma.
\begin{lemma}%
Suppose that $G$ is a plane graph in $\mathcal{G}_{(A)}$, and
$f_{0}$ is the outer face of $G$ with degree at most $11$, then
every proper $3$-coloring of $G[V(f_{0})]$ can be extend to a proper
$3$-coloring of $G$.
\end{lemma}
\begin{proof}%
By way of contradiction, we assume that the result is not true. Let
$G$ be a counterexample to the Lemma with the following condition:
$|V(G)| + |E(G)|$ is minimum among all the counterexamples. Let
$C_{0}$ be the boundary of the outer face $f_{0}$. Then there exists
a proper 3-coloring of $G[V(f_0)]$ which cannot be extended to a
proper 3-coloring of $G$. Moreover, the minimum counterexample $G$
has the following properties.
\begin{listitem}%
The graph $G$ is simple, \ie it has no loops and no multiple edges.
\end{listitem}
\begin{listitem}\label{intC0}
$\int(C_{0})$ contains at least one vertex.
\end{listitem}
\begin{listitem}%
For every vertex $v$ in $\int(C_{0})$, the degree of $v$ in $G$ is at least three.
\end{listitem}
\begin{listitem}%
The graph $G$ is $2$-connected, and thus the boundary of each face is a cycle.
\end{listitem}

From now on, for any integer $i \geq 4$, $i^{-}$ denotes every
positive integer ranges from 3 to $i$ and $i^{+}$ denotes all the
positive integer greater than $i$.

\begin{listitem}\label{NoSeparatCycle}%
The graph $G$ has no separating cycles of length at most eleven. So
every $11^{-}$-cycle is a facial cycle.
\end{listitem}
\begin{listitem}\label{C0noChord}%
The outer cycle $C_{0}$ has no chords. For any inner face $f$ of
$G$, at least one vertex of the boundary of $f$ is not on $C_{0}$.
\end{listitem}
\begin{proof}%
Let $xy$ be a chord of the outer cycle $C_{0}$. By the minimality of
$G$, the $3$-coloring of $G[V(f_{0})]$ can be extend to a proper
$3$-coloring of $G - xy$. Obviously, it is also a proper
$3$-coloring of $G$.
\end{proof}
\begin{listitem}\label{AtMost2Neighbors}
If $C$ is a cycle of length at most $11$, then every vertex in
$\int(C)$ has at most two neighbors on $C$.
\end{listitem}
\begin{proof}%
If $v$ has three neighbors on the cycle $C$, then the vertex $v$ and
its three incident edges partition the cycle into three cycles.
According to the lengths of the smallest cycle, there are several
cases. If the smallest one is of length three, the other two are of
length at least eight as $G \in \mathcal{G}_{(A)}$, then $|C| \geq 3
+ 8 + 8 - 6= 13$, a contradiction. If the smallest one is of length
four, the other two are of length at least seven, then $|C| \geq 4 +
7 + 7 - 6 = 12$, a contradiction. If the smallest one is of length
five, the other two are of length at least seven, then $|C| \geq 5 +
7 + 7 - 6 = 13$, a contradiction. If the smallest one is of length
no less than six, then $|C| \geq 6 + 6 + 6 - 6=12$, a contradiction.
\end{proof}
\begin{listitem}\label{AtMost1Neighbor}%
If $C$ is a cycle of length at most $11$, then every vertex in
$\int(C)$ has at most one neighbor on $C$, except when $|C| = 11$
and the two neighbors on $C$ are consecutive.
\end{listitem}
\begin{proof}%
Suppose that there exists a vertex $v$ in $\int(C)$ such that it has
two neighbors $v_{1}$ and $v_{2}$ on the cycle $C$. By
(\ref{AtMost2Neighbors}), the vertices $v_{1}$ and $v_{2}$ are the
only two neighbors on $C$; and the path $v_{1}vv_{2}$ split the
cycle $C$ into two cycles $C_{1}=vC[v_{1}, v_{2}]v$ and
$C_{2}=vC[v_{2}, v_{1}]v$. Clearly, the vertex $v$ is in
$\int(C_{0})$, so $\deg_{G}(v) \geq 3$ and $v$ has at least one
neighbor in $\int(C)$. Then at least one of $C_i$ $(i=1, 2)$, say
$C_{1}$, is a separating cycle. It follows from
(\ref{NoSeparatCycle}) that $|C_{1}| \geq 12$. Hence $|C_{2}|=3$ and
$|C| = 11$.
\end{proof}

\begin{figure}[!htb]
\centering
\begin{minipage}[b]{.5\textwidth}
\centering
\includegraphics[width=0.7\textwidth]{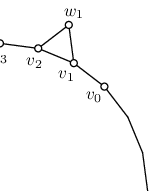}
 \caption{A local structure in (\ref{SimiTetrad})}\label{SimiTetrad-1}
\end{minipage}%
\begin{minipage}[b]{.5\textwidth}
\centering
\includegraphics[width=0.7\textwidth]{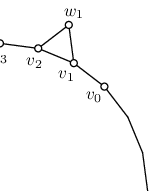}
\caption{A tetrad} \label{Tetrad-1}
\end{minipage}%
\end{figure}

\begin{listitem}\label{SimiTetrad}%
Let $f$ be a face with boundary $\partial(f) = v_{0}v_{1}v_{2} \dots
v_{l}v_{0}$. Assume that $v_{1}, v_{2}, \dots, v_{k}$ (where $k \geq
3$) are inner vertices consecutively on the boundary, and they are
all of degree three. If the edge $v_{1}v_{2}$ is in a triangle
$v_{1}w_{1}v_{2}v_{1}$ and the other neighbor of $v_{3}$ is $w_{2}$,
then the distance between $v_{0}$ and $w_{2}$ in the graph $G -
\{v_{1}, v_{2}, \dots, v_{k} \}$ is at most seven, and $k = 3$, see
\autoref{SimiTetrad-1}. Moreover, vertices $w_2, v_3, v_4$ are
consecutively on the boundary of a $5^{-}$-face.
\end{listitem}
\begin{proof}
Let $G^{*}$ be the graph obtained from $G$ by deleting vertices
$v_{1}, v_{2}, \dots, v_{k}$ and identifying vertex $v_{0}$ with
vertex $w_{2}$.

In the following proof, we will frequently use the fact that $G \in
\mathcal{G}_{(A)}$ and the triangle $v_{1}w_{1}v_{2}v_{1}$ is not
adjacent to any $7^{-}$-cycle.

First, we show that the distance between $v_{0}$ and $w_{2}$ in the
graph $G - \{v_{1}, v_{2}, \dots, v_{k}\}$ is at most seven. If the
distance is greater than seven, then the identification does not
create new cycles of length at most seven, and hence $G^{*} \in
\mathcal{G}_{(A)}$. Moreover, the cycle $C_{0}$ is also the outer
cycle of $G^{*}$, and the identification does not create chords of
$C_{0}$. By the minimality of $G$, the precoloring of $C_{0}$ can be
extend to a proper 3-coloring of $G^{*}$, and then a proper
3-coloring of $G$, a contradiction. So we may assume that the
distance between $v_{0}$ and $w_{2}$ in the graph $G - \{v_{1},
v_{2}, \dots, v_{k}\}$ is at most seven.

Let $P$ be a shortest path between $v_{0}$ and $w_{2}$ in the graph
$G - \{v_{1}, v_{2}, \dots, v_{k}\}$. It is easy to see that $w_{1}$
is not on the path $P$. If $v_{4}$ is not on the path $P$, then the
cycle $Pw_{2}v_{3}v_{2}v_{1}v_{0}$ is a cycle of length at most
eleven separating $w_{1}$ from $v_{4}$. Therefore, the vertex
$v_{4}$ is on the path $P$, and hence $k = 3$. The cycle $P[v_{0},
v_{4}]v_{3}v_{2}v_{1}v_{0}$ have a common edge with the triangle
$v_{1}w_{1}v_{2}v_{1}$, so $|P[v_{0}, v_{4}]| \geq 4$, and then
$|P[w_{2}, v_{4}]| \leq 3$. Therefore, $P[w_{2}, v_{4}]v_{3}w_{2}$
is a cycle of length at most five, by (\ref{NoSeparatCycle}), it
bounds an inner face of degree at most five.
\end{proof}
By (\ref{SimiTetrad}), if $v_{4}v_{3}w_{2}$ is also a $3$-cycle,
then $v_{4}$ is on $C_{0}$ or has degree at least four.

A {\em tetrad} is a local structure having four bad vertices $v_{1},
v_{2}, v_{3}$ and $v_{4}$ consecutively on the boundary of a face
(the degree of the face is at least six) with the edge $v_{1}v_{2}$
in a triangle and the edge $v_{3}v_{4}$ in a triangle (see
\autoref{Tetrad-1}).
\begin{listitem}\label{NoTetrad}
  The graph $G$ contains no tetrad.
\end{listitem}

It follows from (\ref{SimiTetrad}) and (\ref{NoTetrad}) that:
\begin{listitem}\label{No5consec}
 A face doesn't have five bad vertices consecutively on the boundary.
 \end{listitem}

\begin{listitem}%
The graph $G$ has no inner $4$-faces.
\end{listitem}
\begin{proof}%
Suppose that $f$ is an inner $4$-face, and the boundary of $f$ is a
$4$-cycle $\partial(f) =v_{1}v_{2}v_{3}v_{4}v_{1}$ (the $v_{i}$'s
appearing counterclockwise on $f$). Let $G^{*}$ be the graph
obtained from $G$ by identifying the vertices $v_{1}$ with $v_{3}$.

First, we show that the identification does not damage the outer
cycle $C_{0}$. Otherwise, both $v_{1}$ and $v_{3}$ are on the outer
cycle $C_{0}$, by (\ref{C0noChord}), one of $\{v_{2}, v_{4}\}$, say
$v_{2}$, is not on $C_{0}$. Then by (\ref{AtMost1Neighbor}), $v_{2}$
has two neighbors consecutive on $C_{0}$, that is, $v_{1}$ and
$v_{3}$ are adjacent in $G$, contradicting the fact that $4$-cycles
are chordless. Therefore, $C_{0}$ is also the outer cycle of
$G^{*}$.

Assume that the identification create a new chord of $C_{0}$.
Without loss of generality, assume that $v_{3}$ is on $C_{0}$, but
$v_{1}$ is not on $C_{0}$ and $v_{1}$ has a neighbor on $C_{0}$, say
$v$. Since the edge $v_{1}v_{2}$ is in the $4$-cycle
$v_{1}v_{2}v_{3}v_{4}v_{1}$, then it can not be contained in a
$3$-cycle, by (\ref{AtMost1Neighbor}), the vertex $v_{2}$ is in
$\int(C_{0})$. Similarly, the vertex $v_{4}$ is in $\int(C_{0})$.
The cycle $C_{0}[v, v_{3}]v_{4}v_{1}v$ is a separating cycle of $G$,
then $|C_{0}[v, v_{3}]| \geq 9$. Similarly, the cycle $C_{0}[v_{3},
v]v_{1}v_{2}v_{3}$ is a separating cycle of $G$, and $|C_{0}[v_{3},
v]| \geq 9$, then $|C_{0}| \geq 9 + 9 = 18$, a contradiction.

Let $C^{*}$ be an arbitrary new cycle of length at most seven
created by the identification. Then it corresponds to a
$v_{1}$-$v_{3}$ path $P = v_{1}x_{1} \dots x_{k}v_{3}$ in $G$, where
$k \leq 6$. If neither $v_{2}$ nor $v_{4}$ is on the path $P$, then
there must be a separating cycle of length at most nine, a
contradiction to (\ref{NoSeparatCycle}). Hence $v_{2}$ is on the
path $P$, without loss of generality, assume $v_{2} = x_{1}$. If $k
\leq 5$, the cycle $x_{1}x_2 \dots x_{k}v_{3}v_2$ is a cycle of
length at most six, it has a common edge $v_{1}v_{2}$ with the cycle
$v_{1}v_{2}v_{3}v_{4}v_{1}$ in $G$, which is impossible. So,
$|C^{*}| =7$ and $k=6$. Since the cycle $x_{1}x_2 \dots
x_{6}v_{3}v_2$ is not adjacent to any $3$-cycle in $G$ for $G \in
\mathcal{G}_{(A)}$, and hence it is not adjacent to any 3-cycle in
$G^{*}$. Similarly, edges $v_3v_4, v_4v_1, v_1v_2$ is not adjacent
to any 3-cycle in $G^{*}$, because they all lie in a 4-cycle of $G$
and $G \in \mathcal{G}_{(A)}$. Therefore, $C^{*}$ is not adjacent to
any 3-cycle in $G^{*}$ and hence $G^{*} \in \mathcal{G}_{(A)}$.

By the minimality of $G$, the precoloring of $C_{0}$ can be extend
to a proper $3$-coloring of $G^{*}$, which just corresponds to a
proper $3$-coloring of $G$, a contradiction.
\end{proof}

\begin{listitem}%
The graph $G$ has no inner $6$-faces.
\end{listitem}
\begin{proof}
Let $f$ be an inner $6$-face, with boundary a $6$-cycle $\partial(f)
= v_{1}v_{2}v_{3}v_{4}v_{5}v_{6}v_{1}$. Obviously, by
(\ref{C0noChord}), there exists at least one vertex of $\{v_{1},
v_{2}, v_{3}, v_{4}, v_{5}, v_{6}\}$, say $v_1$, is not on $C_{0}$.
By (\ref{AtMost1Neighbor}), either $v_{2}$ or $v_{6}$ is not on
$C_{0}$, we assume that $v_{2}$ is not on $C_{0}$. Let $G^{*}$ be
the graph obtained from $G$ by identifying the vertices $v_{1}$ with
$v_{5}$ and $v_{2}$ with $v_{4}$. Because neither $v_{1}$ nor
$v_{2}$ is on $C_{0}$, the cycle $C_{0}$ is also the outer cycle of
the graph $G^{*}$.

We show that the outer cycle $C_{0}$ has no chord in $G^{*}$.
Otherwise, we assume that there exists a chord in $G^{*}$, without
loss of generality, we assume that $v_{4}$ is on $C_{0}$ and $v_{2}$
is not on $C_{0}$ but it has a neighbor $v$ on $C_{0}$. Because the
edge $v_{2}v_{3}$ is contained in the $6$-cycle
$v_{1}v_{2}v_{3}v_{4}v_{5}v_{6}v_{1}$, it can not be contained in
any $3$-cycle, so $v_{3}$ is in $\int(C_{0})$. By the nonadjacency
condition, the cycle $C_{0}[v, v_{4}] v_{3}v_{2}v$ has length at
least six, and $|C_{0}[v, v_{4}]| \geq 3$. The cycle $C_{0}[v_{4},
v]v_{2}v_{3}v_{4}$ is a separating cycle, and $|C_{0}[v_{4}, v]|
\geq 9$. Hence $|C_{0}| \geq 3 + 9 = 12$, a contradiction.

We can also show that the identification does not make short cycles
of $G^{*}$ adjacent, that is to say, $G^{*} \in \mathcal{G}_{(A)}$.

Now $G^{*}$ is a graph having the nonadjacency property
$\mathcal{A}$ and $G^{*}$ is a smaller graph than $G$, then the
precoloring of $C_{0}$ can be extend to a proper $3$-coloring of
$G^{*}$, which obviously corresponds to a proper $3$-coloring of
$G$, a contradiction.
\end{proof}

\begin{listitem}\label{5Face1}
Suppose that $f$ is a $5$-face with boundary $\partial(f) =
v_{1}v_{2}v_{3}v_{4}v_{5}v_{1}$, and both $v_{1}$ and $v_{3}$ are on
the outer cycle $C_{0}$, then there exists a $7$-cycle $C'$ such
that $E(C') \cap \{v_{3}v_{4}, v_{4}v_{5}, v_{5}v_{1}\} \neq
\emptyset$.
\end{listitem}
\begin{proof}
As $5$-cycles are chordless, vertices $v_{1}$ and $v_{3}$ are not
adjacent in $G$. By (\ref{AtMost1Neighbor}), the vertex $v_{2}$ is
on the cycle $C_{0}$. By (\ref{C0noChord}), edges $v_{1}v_{2}$ and
$v_{2}v_{3}$ are consecutive on $C_{0}$. Hence the vertices $v_{4},
v_{5}$ are in $\int(C_{0})$. By (\ref{AtMost1Neighbor}), the vertex
$v_{4}$ has a neighbor distinct from $v_5$, in $\int(C_{0})$. So the
cycle $C_{0}[v_{3}, v_{1}]v_{5}v_{4}v_{3}$ is a separating cycle,
and then $|C_{0}[v_{3}, v_{1}]| \geq 9$. So, the cycle $C_{0}$ is a
$11$-cycle.

Let $C = v_{1}v_{2}v_{3}x_{4} \dots x_{k}v_{1}$ be a cycle of length
at most nine distinct from the $5$-cycle
$v_{1}v_{2}v_{3}v_{4}v_{5}v_{1}$.  Clearly, $C \neq C_{0}$, there
exists at least one vertex in $\ext(C)$. If one of $\{v_{4},
v_{5}\}$ is not on $C$, the cycle $C$ is a separating cycle of
length at most nine, a contradiction. Therefore, both $v_{4}$ and
$v_{5}$ are on the cycle $C$. The two vertices $v_{4}$ and $v_{5}$
divide the path $C[v_{3}, v_{1}]$ into three segments, at least one
of the three segments is a path with length more than one. By the
nonadjacency condition, this path is of length at least six. Hence
$|C| = 2 + |C[v_{3}, v_{1}]| \geq 2 + 2 + 6 = 10$, a contradiction.
Then every cycle of $G$ containing edge $v_{1}v_{2}$ and
$v_{2}v_{3}$ must be of length at least ten except the cycle
$v_{1}v_{2}v_{3}v_{4}v_{5}v_{1}$. That is, every path linking
$v_{1}$ and $v_{3}$ in $G - \{v_{2}\}$ is of length at least eight
except the path $v_{3}v_{4}v_{5}v_{1}$. \vskip 7pt Case 1: The
vertices $v_{1}$ and $v_{3}$ receives different colors in the
precoloring. \vskip 7pt Delete the vertex $v_{2}$ and its incident
edges $v_{1}v_{2}$ and $v_{2}v_{3}$, add a new edge $v_{1}v_{3}$, we
obtain a new graph $G^{*}$. Obviously, the precoloring of $C_{0}$
corresponds to a proper $3$-coloring of the outer cycle of $G^{*}$.

We next show that $G^{*} \in \mathcal{G}_{(A)}$. If there exist two
cycles $C_{1}$ and $C_{2}$ violates the nonadjacency condition in
$G^{*}$, then one of $\{C_{1}, C_{2}\}$, say $C_{1}$, must contain
the edge $v_{1}v_{3}$. Since the path linking $v_{1}$ and $v_{3}$ in
$G - \{v_{2}\}$ is of length at least eight except the path
$v_{3}v_{4}v_{5}v_{1}$, then $C_{1} = v_{1}v_{3}v_{4}v_{5}v_{1}$ and
the cycle $C_{2}$ does not contain the edge $v_{1}v_{3}$,
consequently, the cycle $C_{2}$ is a cycle of $G - \{v_{2}\}$. By
the violated condition in $G^{*}$, we have $E(C_{2}) \cap
\{v_{3}v_{4}, v_{4}v_{5}, v_{5}v_{1}\} \neq \emptyset$ and $|C_{2}|
\leq 6$. Therefore, the $5$-cycle $v_{1}v_{2}v_{3}v_{4}v_{5}v_{1}$
has a common edge with the cycle $C_{2}$ in $G$, a contradiction.
Hence, $G^{*}$ is a graph having the nonadjacency property
$\mathcal{A}$.

By the minimality of $G$, the precoloring of the outer cycle of
$G^{*}$ can be extend to a proper $3$-coloring of $G^{*}$, which
corresponds to a proper $3$-coloring of $G$, a contradiction. \vskip
7pt Case 2: The vertices $v_{1}$ and $v_{3}$ receive the same color
in the precoloring. \vskip 7pt Delete the vertex $v_{2}$ together
with its incident edges $v_{1}v_{2}$ and $v_{2}v_{3}$, and identify
the vertices $v_{1}$ with $v_{3}$, then we obtain a new graph
$G^{*}$. Obviously, the precoloring of $C_{0}$ corresponds to a
proper 3-coloring of the outer cycle of $G^{*}$.

If $G^{*} \in \mathcal{G}_{(A)}$, the proper 3-coloring of the outer
cycle of $G^{*}$ can be extended to a proper 3-coloring of $G^{*}$,
which corresponds a proper 3-coloring of $G$. So, $G^{*} \notin
\mathcal{G}_{(A)}$. In other words, the identification do violate
the nonadjacency condition. Then there exist two cycles $C_{1}$ and
$C_{2}$ of length at most seven which are adjacent in $G^{*}$. If
both $C_{1}$ and $C_{2}$ are cycles of $G$, this contradicts the
nonadjacency condition in $G$. Thus, there exists a path of length
at most seven linking $v_{1}$ and $v_{3}$ in the graph $G -
\{v_{2}\}$. It must be the path $v_{1}v_{5}v_{4}v_{3}$, because in
the graph $G - \{v_{2}\}$, the path linking $v_{1}$ and $v_{3}$ is
of length at least eight except the path $v_{3}v_{4}v_{5}v_{1}$; and
the other cycle $C'$ is a cycle of $G$ with length seven. Moreover,
$E(C') \cap \{v_{3}v_{4}, v_{4}v_{5}, v_{5}v_{1}\} \neq \emptyset$.

\end{proof}

\begin{listitem}\label{5Face2}
Suppose that $f$ is a $5$-face with boundary $\partial(f) =
v_{1}v_{2}v_{3}v_{4}v_{5}v_{1}$, and either $v_{1}$ or $v_{3}$ is
not on the outer cycle $C_{0}$, then there exists a $7$-cycle
$C^{*}$ such that $E(C^{*}) \cap \{v_{3}v_{4}, v_{4}v_{5},
v_{5}v_{1}\} \neq \emptyset$.
\end{listitem}
\begin{proof}
Without loss of generality, assume that $v_{1}$ is not on the outer
cycle $C_{0}$. Let $G^{*}$ be the graph obtained from $G$ by
identifying the vertices $v_{1}$ with $v_{3}$. Clearly, the
identification dose not damage the outer cycle $C_{0}$.

First, we show that the identification dose not create a chord of
$C_{0}$. Otherwise, the vertex $v_{1}$ has a neighbor $v$ on the
cycle $C_{0}$ and the vertex $v_{3}$ is on the outer cycle $C_{0}$.
By (\ref{AtMost1Neighbor}) and the nonadjacency condition, the
vertices $v_{2}$ and $v_{5}$ are in $\int(C_{0})$. Since the cycle
$C_{0}[v, v_{3}]v_{4}v_{5}v_{1}v$ is a separating cycle of $G$, then
$|C_{0}[v, v_{3}]| \geq 8$. Similarly, the cycle $C_{0}[v_{3},
v]v_{1}v_{5}v_{4}v_{3}$ is a separating cycle of $G$, and
$|C_{0}[v_{3}, v]| \geq 8$. Hence $|C_{0}| \geq 16$, a
contradiction. So the identification does not create a chord of
$C_{0}$, the precoloring of $C_{0}$ is also a proper $3$-coloring of
the outer cycle of $G^{*}$.

If $G^{*}$ is a graph having the nonadjacency property
$\mathcal{A}$, then the precoloring of $C_{0}$ can be extend to a
proper $3$-coloring of $G^{*}$, and the coloring corresponds to a
proper coloring of $G$, a contradiction. Then there exists two
cycles that violate the nonadjacency condition in $G^{*}$. Clearly,
one of them must the triangle $v^{*}v_{4}v_{5}$, where $v^{*}$ is
the vertex obtained by the identifying $v_{1}$ with $v_{3}$, and the
other cycle $C^{*}$ must be a cycle of $G$. By the nonadjacency in
$G$, the cycle $C^{*}$ is a cycle of $G$ with length seven, and it
has a common edge with the triangle $v^{*}v_{4}v_{5}$. That is,
there exists a cycle $C^{*}$ of length seven in the graph $G$ such
that $E(C^{*}) \cap \{v_{3}v_{4}, v_{4}v_{5}, v_{5}v_{1}\} =
\emptyset$.
\end{proof}

Finally, we use the discharging method to get a contradiction and
finish the proof of the lemma.

The Euler formula: for the plane graph $G$, $|V(G)| - |E(G)| +
|F(G)| = 2$, can be written as following:

$$\sum_{v \in V(G)}(\deg_{G}(v) - 4) + \sum_{f \in F(G)}(\deg(f) - 4) = - 8.$$

Initially, set the charge of every vertex $v \in V(G)$ by $w(v) =
\deg_{G}(v) - 4$, and the charge of every face $f \neq f_{0}$ by
$w(f) = \deg(f) - 4$ and $w(f_{0}) = \deg(f_{0}) + 4$. Obviously,
the total sum of the initial charges is zero, \ie $$\sum_{x \in V(G)
\cup F(G)} w(x) = 0.$$

The discharging rule:
\begin{enumerate}[(R1)]
\item Each inner $3$-face receives charge $1/3$ from each incident vertex.
\item If $\deg_{G}(v) = 5$, the vertex $v$ sends charge $1/15$ to each incident $7^{+}$-face.
\item If $\deg_{G}(v) \geq 6$, the vertex $v$ sends charge $1/3$ to each incident face.
\item For all the inner vertices $v$:
\begin{enumerate}[(a)]
  \item If $\deg_{G}(v) = 3$ and $v$ is incident with a $3$-face, then $v$ receives charge $2/3$ from each incident non-triangular face;
  \item If $\deg_{G}(v) = 3$ and $v$ is incident with a $5$-face, then $v$ receives charge $1/5$ from the $5$-face and receives charge $2/5$ from each non-$5$-face;
  \item If $\deg_{G}(v) = 3$ and $v$ is not incident with $3$- or $5$-faces, then $v$ receives charge $1/3$ from each incident face;
  \item If $\deg_{G}(v) = 4$ and $v$ is incident with exactly one $3$-face, but not incident to any $5$-face, then $v$ receives charge $1/3$ from the incident face non-adjacent to the $3$-face;
    \item If $\deg_{G}(v) = 4$ and $v$ is incident with only one $3$-face, $v$ is incident to a $5$-face, then $v$ receives charge $1/15$ from the incident face adjacent to the $3$-face, receives charge $1/5$ from the $5$-face;
  \item If $\deg_{G}(v) = 4$ and $v$ is incident with two $3$-face, then $v$ receives charge $1/3$ from each incident non-triangular face;
  \item If $\deg_{G}(v) = 4$ and $v$ is not incident with $3$-face, but it is incident with $5$-faces, then $v$ receives charge $1/5$ from each $5$-face and sends charge $1/15$ to each incident $5^{+}$-face.
\end{enumerate}
\item For all the outer vertices $v$:
\begin{enumerate}[(a)]
  \item if $\deg(v) = 2$ and $v$ is incident with an inner $5$-face, then $v$ receives charge $3/5$ from the inner $5$-face and receives charge $7/5$from the outer face;
  \item if $\deg(v) = 2$ and $v$ is incident with an inner face having degree at least seven, then $v$ receives charge $2/3$ from the inner face and receives charge $4/3$ from the outer face;
  \item if $\deg(v) = 3$, then the vertex $v$ receives charge $4/3$ from the outer face;
  \item if $\deg(v) = 4$, then the vertex $v$ receives charge $2/3$ from the outer face;
\end{enumerate}

\end{enumerate}

\begin{figure}[!htb]
\centering
  \begin{minipage}[b]{0.3\textwidth}
   \centering
   \subfigure[The discharging rule (R1)]{\includegraphics[width=1in]{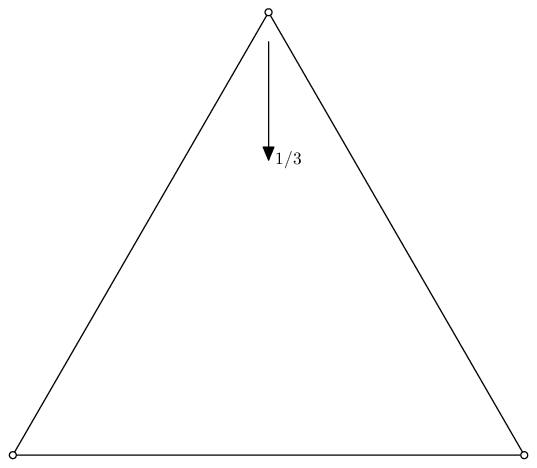}}
  \end{minipage}%
    \hspace{0.04\textwidth}%
  \begin{minipage}[b]{0.3\textwidth}
   \centering
   \subfigure[The discharging rule (R4a)]{\includegraphics[width=1in]{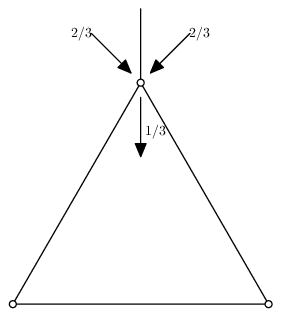}}
  \end{minipage}%
   \hspace{0.04\textwidth}%
  \begin{minipage}[b]{0.3\textwidth}
   \centering
   \subfigure[The discharging rule (R4b)]{\includegraphics[width=1in]{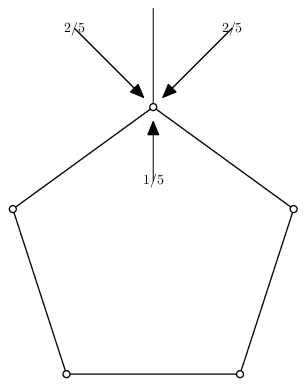}}
  \end{minipage}\\[20pt]
  \begin{minipage}[b]{0.3\textwidth}
   \centering
   \subfigure[The discharging rule (R4d)]{\includegraphics[width=1in]{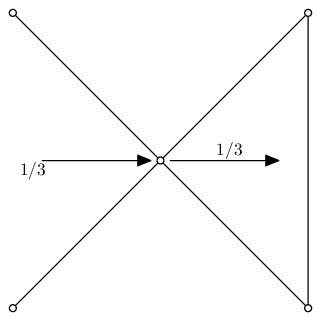}}
  \end{minipage}%
   \hspace{0.04\textwidth}%
  \begin{minipage}[b]{0.3\textwidth}
   \centering
   \subfigure[The discharging rule (R4e)]{\includegraphics[width=1in]{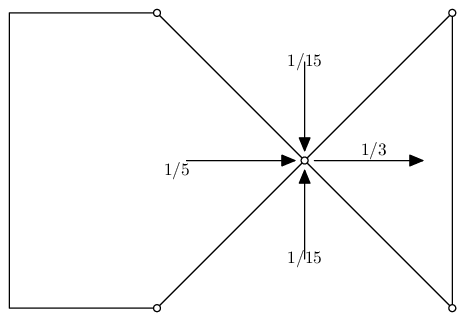}}
  \end{minipage}%
   \hspace{0.04\textwidth}%
  \begin{minipage}[b]{0.3\textwidth}
   \centering
   \subfigure[The discharging rule (R4f)]{\includegraphics[width=1in]{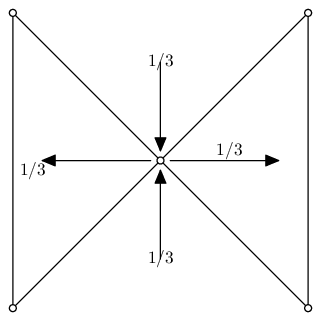}}
  \end{minipage}\\[20pt]
  \begin{minipage}[b]{0.3\textwidth}
   \centering
   \subfigure[The discharging rule (R4g)]{\includegraphics[width=1in]{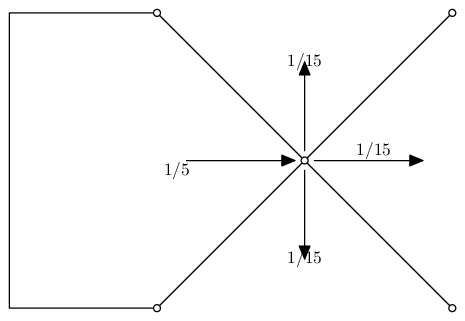}}
  \end{minipage}%
   \hspace{0.04\textwidth}%
  \begin{minipage}[b]{0.3\textwidth}
   \centering
   \subfigure[The discharging rule (R5a)]{\includegraphics[width=1in]{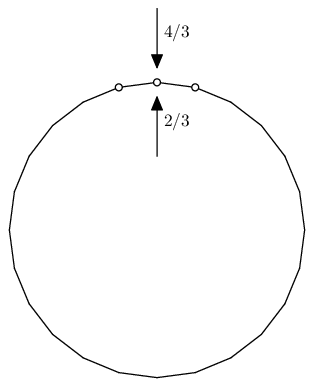}}
  \end{minipage}%
   \hspace{0.04\textwidth}%
  \begin{minipage}[b]{0.3\textwidth}
   \centering
   \subfigure[The discharging rule (R5b)]{\includegraphics[width=1in]{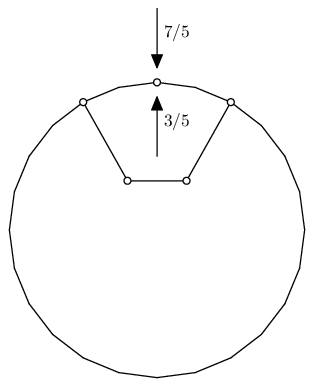}}
  \end{minipage}\\[20pt]
  \caption{The discharging process.}
\end{figure}

\begin{listitem}
  After the discharging process, all the vertices have nonnegative final charges.
\end{listitem}

\begin{figure}%
\centering
\includegraphics{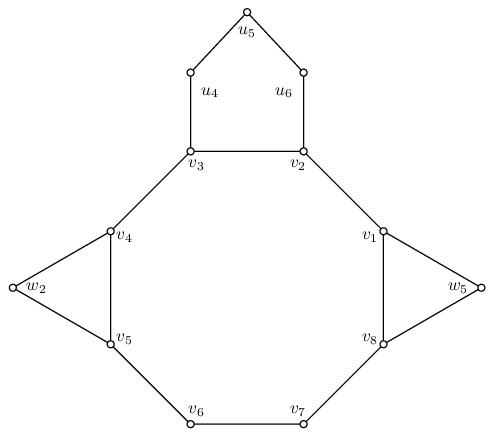}
\caption{A big face is incident with a $5$-face.}
\label{FiveFace}
\end{figure}

\begin{remark}
By the discharging rule, if a face $f$ sends charge $2/5$ to its
incident vertex $v_{3}$, then the vertex $v_{3}$ has degree three,
and it is incident with a $5$-face, see \autoref{FiveFace}. If
$\deg_{G}(v_{2}) \geq 4$, then the face $f$ sends to the vertex
$v_{2}$ at most $1/15$. If $\deg_{G}(v_{2}) = 3$, the face $f$ sends
charge $2/5$ to the vertex $v_{2}$, and then it follows from
(\ref{5Face1}, \ref{5Face2}) and the fact $G \in \mathcal{G}_{(A)}$
that either $v_{4}$ or $v_{1}$ is not bad; note that in this case
three non-bad vertices are consecutively on the face boundary.
\end{remark}
\begin{listitem}
For all the face $f$, the final charge of $f$ is nonnegative.
Moreover, the final charge of the outer face is positive.
\end{listitem}
\begin{proof}
Consider the outer face $f_{0}$. Assume that there are $l$ outer
vertices receiving charge $7/5$ from the outer face. Obviously, $l
\leq 5$. Therefore, the final charge of $f_{0}$ is at least
$\deg(f_{0}) + 4 - \frac{7}{5}l - \frac{4}{3}(\deg(f_{0}) - l) =
-\frac{1}{3}\deg(f_{0}) + 4 - \frac{1}{15}l > 0$.

If $f$ is an inner $3$-face, then the final charge of $f$ is at
least $3 - 4 + 3 \times \frac{1}{3}  = 0$.

If $f$ is an inner $5$-face, and the boundary of $f$ contains a
$2$-vertex, then the face sends nothing to two incident vertices,
see Fig. 3(i), the final charge of $f$ is at least $5 - 4 -
\frac{3}{5} - 2 \times \frac{1}{5} = 0$.

If $f$ is an inner $5$-face, and the boundary of $f$ contains no
$2$-vertices, then the final charge of $f$ is at least $5 - 4 - 5
\times \frac{1}{5} = 0$.

Let $f$ be an inner $7$-face. By (\ref{AtMost1Neighbor}) and the
hypothesis that $3$-cycles are not adjacent to $7$-cycles, the
boundary of $f$ contains at least two vertices in $\int(C_{0})$, and
the face $f$ send to each such vertex by at most $2/5$.

If $f$ is an inner $7$-face which is not incident with a $2$-vertex,
then the final charge of $f$ is at least $7 - 4 - 7 \times
\frac{2}{5} = \frac{1}{5} > 0$; if $f$ is an inner $7$-face which is
incident with a $2$-vertex, then $f$ sends nothing to at least two
vertices on $C_{0}$, and hence the final charge of $f$ is at least
$7 - 4 - 3 \times \frac{2}{3} - 2 \times \frac{2}{5} = \frac{1}{5} >
0$.

Let $f$ be an inner face with degree at least eight. If the face $f$
is incident with a $2$-vertex, it sends nothing to at least two
vertices on $C_{0}$. Thus the final charge of $f$ is at least
$\deg(f) - 4 - \frac{2}{3} (\deg(f) - 2) \geq 0$. Now we assume that
the boundary of an arbitrary inner face with degree at least eight
contains no $2$-vertices. Hence if a face sends a $2/3$ to its
incident vertex, the vertex must be an inner bad vertex.

Let $f$ be an inner face with degree at least ten. It contains at
most $\deg(f) - 2$ bad vertices by (\ref{No5consec}). If the face
$f$ does not send $2/5$ to its incident vertex, then the final
charge of $f$ is at least $\deg(f) - 4 - \frac{2}{3} (\deg(f) - 2) -
2 \times \frac{1}{3} \geq 0$. If the face $f$ send $2/5$ to its
incident vertex, then there are at most $\deg(v) - 3$ bad vertices
on the boundary by (\ref{SimiTetrad}), the final charge of $f$ is at
least $\deg(f) - 4 - \frac{2}{3}(\deg(v) - 3) - 3 \times
\frac{2}{5}\geq \frac{2}{15} > 0$.

Then we only have to consider the inner $8$-faces and $9$-faces.

Let $f$ be an inner $9$-face. By (\ref{No5consec}), the boundary of
face $f$ contains at most seven bad vertices. If the boundary of $f$
contains seven bad vertices, then the other two vertices separate
the seven bad vertices as $4 + 3$ by (\ref{No5consec}), and the four
bad vertices does not form a tetrad by (\ref{NoTetrad}). The local
structure must be as in \autoref{FourPlusThree-1}, and then the
final charge of $f$ is at least $9 - 4 - 7 \times \frac{2}{3} -
\frac{1}{3} = 0$.

  \begin{figure}[!htb]
   \centering
   \includegraphics{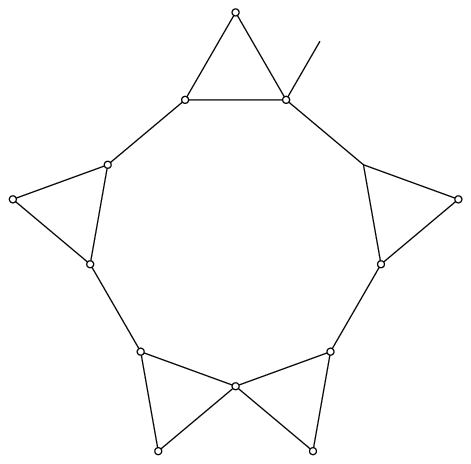}\\
   \caption{}
   \label{FourPlusThree-1}
  \end{figure}

If the boundary of $f$ contains six bad vertices and $f$ does not
send $2/5$ to its incident vertices, then the final charge of $f$ is
at least $9 - 4 - 6 \times \frac{2}{3} - 3 \times \frac{1}{3} = 0$.
If the boundary of $f$ contains six bad vertices and $f$ sends
charge $2/5$ to a vertex, then the final charge of $f$ is at least
$9 - 4 - 6 \times \frac{2}{3} - \frac{2}{5} - \frac{1}{3} -
\frac{1}{15} > 0$ by Remark 1 and (\ref{SimiTetrad}). If the
boundary of $f$ contains at most five bad vertices, then the final
charge of $f$ is at least $9 - 4 - 5 \times \frac{2}{3} - 4 \times
\frac{2}{5} > 0$. \vskip 10pt

Finally, we dealt with the $8$-face $f$. By (\ref{No5consec}), the
boundary of $f$ contains at most six bad vertices.

If the boundary of $f$ contains six bad vertices, then the other two
vertices separate the six bad vertices as $4 + 2$ or $3 + 3$ by
(\ref{No5consec}). Note that these two non-bad vertices are not
consecutively, so the face doesn't sends $2/5$ to the two non-bad
vertices.

\begin{enumerate}[(i)]
\item The six bad vertices are separated by the other two vertices
into two segments, where one contains four bad vertices and the
other contains two bad vertices.

The four bad vertices $v_{1}, v_{2}, v_{3}, v_{4}$ does not form a
tetrad, then $v_{1}v_{8}$ and $v_{4}v_{5}$ are in triangles. If
$v_{6}v_{7}$ is in a triangular face, then $f$ will send nothing to
the vertices $v_{5}$ and $v_{8}$, its final charge is at least $8 -
4 - 6 \times \frac{2}{3} = 0$. If $v_{5}v_{6}, v_{7}v_{8}$ are
respectively in a triangular face, the face $f$ is a two-ear face,
see \autoref{TwoEarFace}, a contradiction(the detail is leaving for
the reader, you can also see \cite{Borodin2005}).
\begin{figure}[!htb]
\centering
\begin{minipage}[b]{.5\textwidth}
\centering
\includegraphics[width=0.7\textwidth]{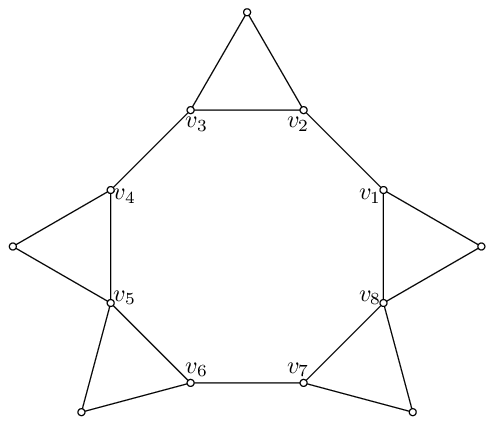}
 \caption{Two ear face}\label{TwoEarFace}
\end{minipage}%
\begin{minipage}[b]{.5\textwidth}
\centering
\includegraphics[width=0.7\textwidth]{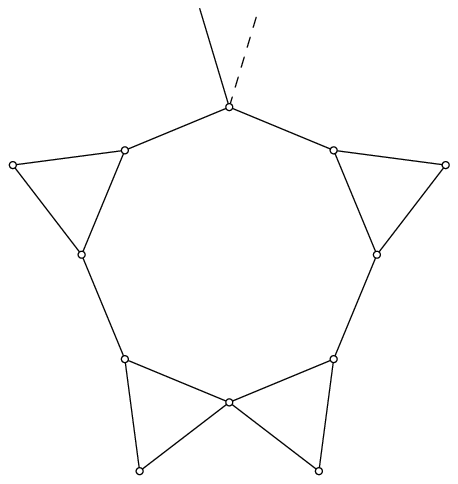}
\caption{One ear face} \label{OneEarFace}
\end{minipage}%
\end{figure}

\item The six bad vertices are separated by the other two vertices into two segments, each of which contains three bad vertices.

It is not too hard to see that the local structure is a one-ear face, see \autoref{OneEarFace}, a contradiction.
\end{enumerate}

Suppose now the boundary of $f$ contains five bad vertices. First,
assume that $f$ sends  $2/5$ to its incident vertex $v_{3}$, see
\autoref{FiveFace}. There exists a vertex $v_{2}$ on $\partial(f)$,
such that $v_{3}v_{2}$ is in a $5$-cycle. If $\deg(v_{2}) \geq 5$,
then the face $f$ receives from the vertex $v_{2}$ at least $1/15$.
Hence the final charge of $f$ is at least $8 - 4 - 5 \times 2/3 -
2/5 + 1/15 - 1/3 = 0$. If $\deg(v_{2}) = 3$, then the three non-bad
vertices are consecutively on the boundary by Remark 1, and the five
bad vertices lie consecutively on the boundary, a contradiction to
(\ref{No5consec}).

Then we assume that $\deg_{G}(v_{2}) = 4$. By (\ref{No5consec}),
vertices $v_1, v_{4}$ are bad and the edge $v_{4}v_{5}$ is in a
triangle. Then, it follows from (\ref{5Face1}, \ref{5Face2}) that
the edge $v_{1}v_{8}$ is in a triangle. By (\ref{SimiTetrad}), the
non-bad vertex is one of $\{v_{6}, v_{7}\}$. But the edge
$v_{6}v_{7}$ is in a triangle, so the non-bad vertex is of degree at
least four. By the discharging rule, the face $f$ sends nothing to
the non-bad vertex. Hence the final charge of the face is at least
$8 - 4 - 5\times 2/3 - 2/5 > 0$.

Then assume $f$ does not send charge $2/5$ to its incident vertices.
If $f$ sends nothing to at least one vertex, the final charge of the
face is at least $8 - 4 - 5 \times 2/3 -2 \times 1/3 >0 $. If not,
by the discharging rules, the bad vertices are paired linked by the
edges in the triangle, a contradiction to the fact that $5$ is odd.

In the end, we may assume the boundary of $f$ contains four bad
vertices, because the final charge of $f$ is no less than $8-4-3
\times 2/3 -5 \times 2/5 =0$ if $f$ contains at most three bad
vertices. If $f$ does not send charge $2/5$ to its incident
vertices, then the final charge of $f$ is at least $8 - 4 - 4 \times
2/3 - 4 \times 1/3 = 0$. Then we assume that $f$ does sends charge
$2/5$ to its incident vertex $v_{3}$, and the edge $v_{2}v_{3}$ is
in a $5$-face. If there exists a non-bad vertex receiving from $f$
at most $1/15$, then the final charge of $f$ is at least $8 - 4 - 4
\times 2/3 - 1/15 - 3 \times 2/5 > 0$. Then we assume that every
non-bad vertex receives charge from $f$ greater than $1/15$, then
$\deg_{G}(v_{2}) = 3$, or $v_2$ receives charge no more than $1/15$
from $f$. By (\ref{SimiTetrad}), one of $\{v_4, v_5\}$ is a non-bad
vertex, and one of $\{v_8, v_1\}$ is a non-bad vertex. By
(\ref{5Face1}, \ref{5Face2}), without loss of generality, we assume
that the $v_{3}v_{4}$ is in a $7$-face, then $v_{4}$ is not bad and
$v_{5}$ is bad. Moreover, $v_5v_6$ is in a triangular. Face $f$
sends charge great than $1/15$ to $v_{4}$, by the discharging rule,
$v_{4}$ is of degree three, but this contradicts (\ref{SimiTetrad}).
\end{proof}
We complete the proof of the color extension lemma.
\end{proof}
\begin{proof}[Proof of Theorem \ref{MainResult}]
Suppose the theorem is not correct. Let $G$ be a minimum
counterexample. Then $G$ is simple, 2-connected, and with girth less
than six. Hence, it must has a cycle $C_{0}$ of length less than
six. If $C_{0}$ is an outer cycle of $G$, a contradiction to the
extension Lemma. If $C_{0}$ is a separating cycle, we can first
color the cycle $C_{0}$, and thus extend the coloring to
$\int(C_{0})$ and $\ext(C_{0})$, and yields a proper $3$-coloring of
$G$, a contradiction. If $C_{0}$ is a inner facial cycle, then we
can redraw the graph $G$, such that $C_{0}$ is the outer cycle, then
apply the extension Lemma, a contradiction.
\end{proof}

\section*{Acknowledgment}
The author is indebted to Dr. Zefang Wu and Prof. Qinglin Yu for
their  careful  reading of the manuscript.


\begin{thebibliography}{10}

\bibitem{Abbott1991}
H.~L. Abbott and B.~Zhou, On small faces in $4$-critical planar graphs, Ars
  Combin. 32 (1991) 203--207.

\bibitem{Borodin1996a}
O.~V. Borodin, Structural properties of plane graphs without adjacent triangles
  and an application to 3-colorings, J. Graph Theory 21 (1996)~(2) 183--186.

\bibitem{Borodin1996b}
O.~V. Borodin, To the paper: ``on small faces in $4$-critical planar graphs''
  [{Ars Combin.} 32 (1991), 203--207] by {H. L. Abbott and B. Zhou}, Ars
  Combin. 43 (1996) 191--192.

\bibitem{Borodin2006}
O.~V. Borodin, A.~N. Glebov, T.~R. Jensen and A.~Raspaud, Planar graphs without
  triangles adjacent to cycles of length from $3$ to $9$ are $3$-colorable,
  Sib. \`{E}lektron. Mat. Izv. 3 (2006) 428--440.

\bibitem{Borodin2009}
O.~V. Borodin, A.~N. Glebov, M.~Montassier and A.~Raspaud, Planar graphs
  without $5$- and $7$-cycles and without adjacent triangles are $3$-colorable,
  J. Combin. Theory Ser. B 99 (2009)~(4) 668--673.

\bibitem{Borodin2005}
O.~V. Borodin, A.~N. Glebov, A.~Raspaud and M.~R. Salavatipour, Planar graphs
  without cycles of length from $4$ to $7$ are $3$-colorable, J. Combin. Theory
  Ser. B 93 (2005)~(2) 303--311.

\bibitem{Borodin2010}
O.~V. Borodin, M.~Montassier and A.~Raspaud, Planar graphs without adjacent
  cycles of length at most seven are $3$-colorable, Discrete Math. 310
  (2010)~(1) 167--173.

\bibitem{Borodin2003}
O.~V. Borodin and A.~Raspaud, A sufficient condition for planar graphs to be
  $3$-colorable, J. Combin. Theory Ser. B 88 (2003)~(1) 17--27.

\bibitem{Garey1976}
M.~R. Garey, D.~S. Johnson and L.~Stockmeyer, Some simplified {NP}-complete
  graph problems, Theoretical Computer Science 1 (1976)~(3) 237--267.

\bibitem{Grotzsch1959}
H.~Gr\"{o}tzsch, Ein dreifarbenzatz fur dreikreisfreie netze auf der kugel,
  Math.-Natur. Reihe 8 (1959) 109--120.

\bibitem{Salavatipour2002}
M.~R. Salavatipour, The three color problem for planar graphs, Tech. Rep.
  CSRG-458, Department of Computer Science, University of Toronto (2002).

\bibitem{Sanders1995}
D.~P. Sanders and Y.~Zhao, A note on the three color problem, Graphs Combin. 11
  (1995)~(1) 91--94.

\bibitem{Xu2006}
B.~Xu, On $3$-colorable plane graphs without $5$- and $7$-cycles, J. Combin.
  Theory Ser. B 96 (2006)~(6) 958--963.

\end{thebibliography}
\end{document}